\documentclass[a4paper]{article}
\usepackage{amsmath}
\usepackage{amsthm}
\usepackage{amssymb}
\usepackage{graphicx}
\usepackage{authblk}
\usepackage{hyperref}
\usepackage{multirow}
\usepackage{cite}
\usepackage{comment}

\newtheorem{theorem}{Theorem}[section]
\newtheorem{proposition}[theorem]{Proposition}

\newtheorem{definition}{Definition}
\newtheorem{remark}{Remark}

\DeclareMathOperator{\divergenceOperator}{div}
\DeclareMathOperator{\cycl}{Cycl}

\newcommand{\sgn}{\mathop{\mathrm{sgn}}}

\newcommand{\includegraph}[2][]{\ifnum\pdfoutput=0\includegraphics[#1]{#2.eps}\else\includegraphics[#1]{#2.pdf}\fi}

\author[1]{Renato Huzak
\footnote{Corresponding author, {\tt renato.huzak@uhasselt.be}}}
\author[2]{Kristian Uldall Kristiansen}
\author[3]{Goran Radunovi\'{c}}

\affil[1]{Hasselt University, Campus Diepenbeek, Agoralaan Gebouw D, 3590 Diepenbeek, Belgium}
\affil[2]{Department of Applied Mathematics and Computer Science, Technical University of Denmark, 2800 Kgs. Lyngby, Denmark}
\affil[3]{University of Zagreb, Faculty of Science, Horvatovac 102a, 10000 Zagreb, Croatia}

\title{Slow divergence integral in regularized piecewise smooth systems}

\date{}

\begin{document}
\maketitle

\begin{abstract}
In this paper we define the notion of slow divergence integral along sliding segments in regularized planar piecewise
smooth systems. The boundary of such segments may contain diverse tangency points. We show that the slow divergence integral is invariant under smooth equivalences. This is a natural generalization of the notion of slow divergence integral along normally hyperbolic portions of curve of singularities in smooth planar slow-fast systems. We give an interesting application of the integral in a model with visible-invisible two-fold of type \textbf{$VI_3$}. It is related to a connection between so-called Minkowski dimension of bounded and monotone ``entry-exit" sequences and the number of sliding limit cycles produced by so-called canard cycles.
\end{abstract}
\textit{Keywords:} limit cycles; piecewise smooth systems; regularization function; slow divergence integral; Minkowski dimension \newline

\tableofcontents

\section{Introduction}\label{Section-Introduction}
The notion of slow divergence integral \cite[Chapter 5]{DDR-book-SF} has proved to be an important tool to study lower and upper bounds of limit cycles in smooth planar slow-fast systems (see e.g. \cite{DDMoreLC,DDR-book-SF,SDICLE1,DPR,1996} and references therein).
In this paper by ``smooth'', we mean differentiable of class $C^\infty$. One of the main goals of this paper is to define the slow divergence integral in regularized planar piecewise smooth (PWS) systems with sliding and to prove its invariance under smooth equivalences (by smooth equivalence we mean smooth coordinate change and a multiplication by a smooth positive function). This is a natural generalization of \cite{RHKK} where the slow divergence integral is defined only for a PWS two-fold bifurcation of type visible-invisible called $VI_3$ and the switching boundary is a straight line (for more details about two-fold singularity $VI_3$ see \cite{Kuznetsov2003} and Sections \ref{section2general} and \ref{sectionVI3}). In this paper we define the slow divergence integral along sliding segments with a regular sliding vector field \cite{filippov1988differential} (Section \ref{sectionSDI-invariance}), and extend it to tangency points when only one vector field is tangent to switching boundary (Section \ref{section-SDI1Tang}), two-fold singularities of sliding type $VV_1$, $II_1$, $VI_2$ and $VI_3$ \cite{Kuznetsov2003}, and to a visible-invisible two-fold singularity when the sliding vector field vanishes at the two-fold point (Section \ref{section-SDI2Fold}).

\smallskip

Consider a smooth planar slow-fast system  
$$X_{\epsilon,\lambda}=f_\lambda Y_\lambda+\epsilon Q_\lambda+O(\epsilon^2)$$
defined on open set $V\subset \mathbb R^2$,
where $0<\epsilon\ll 1$ is the singular perturbation parameter, $\lambda\sim\lambda_0\in\mathbb R^l$, $f_\lambda$ is a smooth family of functions and $Y_\lambda$ and $Q_\lambda$ are  smooth families of vector fields. We suppose that $X_{0,\lambda}$ has a curve of singularities $C_\lambda$ for all $\lambda\sim\lambda_0$ (Fig. \ref{fig:fastfoliation}). We further assume that $\nabla f_\lambda(p)\ne (0,0)$ for all $p\in \{(x,y)\in V \ | \ f_\lambda
(x,y)=0\}$ and that $Y_\lambda$ has no singularities for each $\lambda\sim\lambda_0$. Then we have $C_\lambda=\{f_\lambda=0\}$ and $C_\lambda$ is a smooth family of $1$-dimensional manifolds. 
\smallskip

The orbits of the flow of $X_{0,\lambda}$ are located inside the leaves of  a smooth $1$-dimensional foliation $\mathcal F_\lambda$ on $V$ tangent to $Y_\lambda$ ($\mathcal F_\lambda$ is called the fast foliation of $X_{0,\lambda}$). If $p\in C_\lambda$, then $DX_{0,\lambda}(p)$ has one eigenvalue equal to zero, with eigenspace $T_pC_\lambda$, and the other one equal to $\divergenceOperator X_{0,\lambda}(p)$ (i.e., the trace of $DX_{0,\lambda}(p)$) with eigenspace $T_pl_{\lambda,p}$ ($l_{\lambda,p}$ is the leaf through $p$). We say that $p\in C_\lambda$ is normally hyperbolic if $\divergenceOperator X_{0,\lambda}(p)\ne 0$ (attracting when $\divergenceOperator X_{0,\lambda}(p)<0$ and repelling when $\divergenceOperator X_{0,\lambda}(p)>0$). We can define the notion of slow vector field on normally hyperbolic segments of $C_\lambda$. Let $p\in C_\lambda$ be a  normally hyperbolic singularity and let $\bar Q_\lambda(p)\in T_pC_\lambda$ be the projection of $Q_\lambda(p)$ on $T_pC_\lambda$ in the direction of $T_pl_{\lambda,p}$. $\bar Q_\lambda$ is called the slow vector field and its flow the slow dynamics. The time variable of the slow dynamics is
the slow time $\bar t = \epsilon t$ where $t$ is the time variable of the flow of $X_{\epsilon,\lambda}$. We point out that the classical definition of the slow vector field using center manifolds is equivalent to this definition. For more details see \cite[Chapter 3]{DDR-book-SF}.
\begin{figure}[htb]
	\begin{center}
		\includegraphics[width=5.4cm,height=3.4cm]{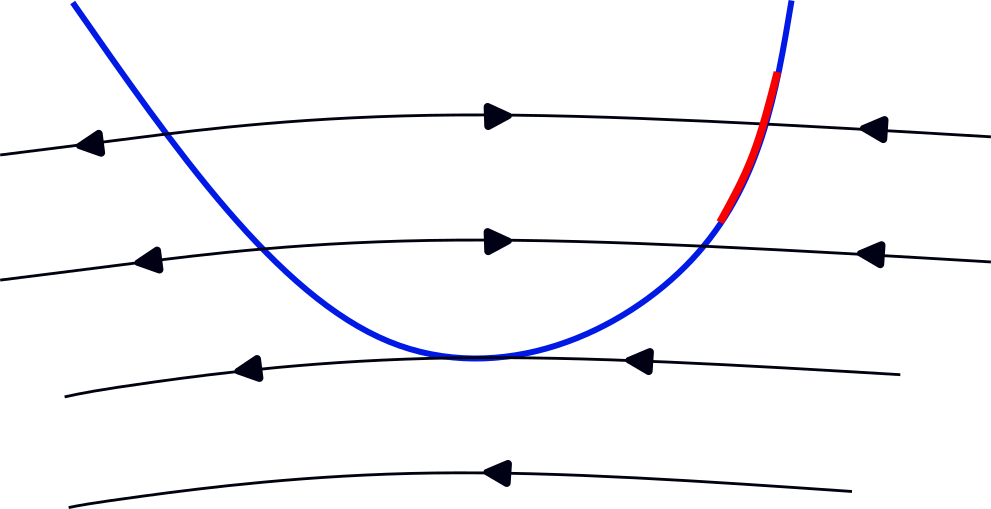}
		{\footnotesize
\put(-52,65){$m_{\lambda}$}
\put(-138,95){$C_{\lambda}$}
\put(-39,18){$\mathcal F_\lambda$}
  }
         \end{center}
	\caption{The dynamics of $X_{0,\lambda}$ with the curve of singularities $C_\lambda$ (blue) and the fast foliation $\mathcal F_\lambda$. A normally hyperbolic segment $m_\lambda\subset C_\lambda$ (red) along which the slow divergence integral can be defined. }
	\label{fig:fastfoliation}
\end{figure}

We define now the notion of slow divergence integral (see \cite[Chapter 5]{DDR-book-SF}). If $m_\lambda\subset C_\lambda$ is a normally hyperbolic segment not containing singularities of  $\bar Q_\lambda$, then the slow divergence integral along $m_\lambda$ is defined by
\begin{equation}
    \label{SDI-classicaldef}
    I(m_\lambda)=\int_{\bar t_1}^{ \bar t_2}\divergenceOperator X_{0,\lambda}(z_\lambda(\bar t))d \bar t
\end{equation}
where $z_\lambda:[\bar t_1,\bar t_2]\to\mathbb R^2$, $z_\lambda'(\bar t)=\bar Q_\lambda(z_\lambda(\bar t))$ and $z_\lambda(\bar t_1)$ and $z_\lambda(\bar t_2)$ are the end points of $m_\lambda$ (we parameterize $m_\lambda$ by $\bar t$). This definition is independent of the choice of parameterization $z_\lambda$ of $m_\lambda$ and the slow divergence integral is invariant under smooth equivalences (see \cite[Section 5.3]{DDR-book-SF}).
\smallskip

If both eigenvalues of the linear part of $X_{0,\lambda}$ at $p\in C_\lambda$ are zero, then we say that $p$ is a (nilpotent) contact point between $C_\lambda$ and $\mathcal F_\lambda$. The slow divergence integral can also be defined along parts of $C_\lambda$ that contain contact points, using its invariance under smooth equivalences and normal forms near contact points (see \cite[Section 5.5]{DDR-book-SF}).
\smallskip

We come now to a natural question: can we define the notion of slow divergence integral if we replace the slow-fast system $X_{\epsilon,\lambda}$ with a regularized planar PWS system? In Section \ref{section2general} we give a positive answer to the question. Instead of $X_{0,\lambda}$ we consider a $\lambda$-family of planar PWS systems \eqref{zZpminvariance} defined in Section \ref{section2general}. The switching boundary $\Sigma_\lambda$ defined after \eqref{zZpminvariance} plays the role of the curve of singularities $C_\lambda$, and the Filippov sliding vector field $Z_\lambda^{sl}$ on sliding subsets of $\Sigma_\lambda$ (see  \eqref{FilippovSVFGeneral}) plays the role of the slow vector field $\bar Q_\lambda$ on normally hyperbolic portions of $C_\lambda$ (see Proposition \ref{prop-sliding}). The function that will be integrated (Definition \ref{definition-SDIGen} in Section \ref{sectionSDI-invariance}) is the divergence of a smooth slow-fast vector field visible in the scaling chart of a cylindrical blow-up applied to regularized PWS system \eqref{regul-gener} (for more details see \cite{RHKK} and the proof of Proposition \ref{prop-sliding}). The notion of slow divergence integral in the PWS setting is well-defined when the sliding vector field $Z_\lambda^{sl}$ has no singularities (see Definition \ref{definition-SDIGen}). 
\smallskip

We show that the slow divergence integral from Definition \ref{definition-SDIGen} is invariant under smooth equivalences (see Theorem \ref{theorem-invariance} in Section \ref{sectionSDI-invariance}).
\smallskip

\begin{figure}[htb]
	\begin{center}
		\includegraphics[width=9.8cm,height=3.7cm]{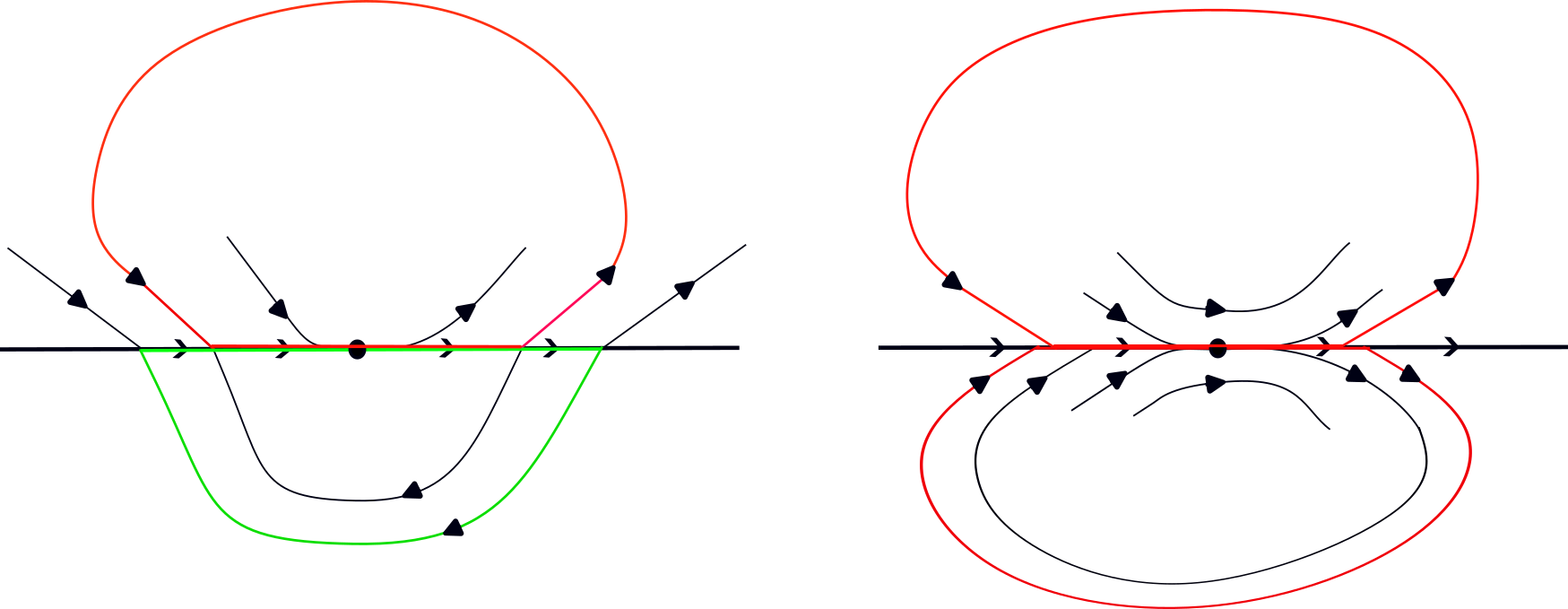}
		{\footnotesize
\put(-282,95){$VI_3$}
\put(-135,95){$VV_1$}
  }
         \end{center}
	\caption{Limit periodic sets in planar PWS systems through two-fold points with sliding (the $VI_3$ case and the $VV_1$ case). They can be located in a region with invisible fold point (green) or in a region with visible fold point (red).}
	\label{fig:2examples1}
\end{figure}

In Sections \ref{section-SDI2Fold} and \ref{section-SDI1Tang} we define the slow divergence integral near tangency points, as already mentioned above (tangency points in $\Sigma_\lambda$ play the role of contact points between $C_\lambda$ and $\mathcal F_\lambda$). We use the invariance of the slow divergence integral under smooth equivalence. The extension of the slow divergence integral to tangency points has proved to be crucial when we study the number of sliding limit cycles (i.e., isolated closed orbits with sliding segments) of a regularized planar PWS system produced by so-called canard limit periodic sets or canard cycles (for more details see \cite{RHKK}). In \cite{RHKK} only the $VI_3$ case has been studied, with canard cycles located in the region with invisible fold point (see the green closed curve in Fig. \ref{fig:2examples1}). Canard cycles contain both stable and unstable sliding portions of the discontinuity manifold (often called switching boundary). A canard cycle can also be located in a region with visible fold point (for example, red closed curves in Fig. \ref{fig:2examples1}), and again the slow divergence integral associated to the segment of the switching
boundary contained in the canard cycle plays an important role when studying sliding limit cycles (see \cite{RHKK1}).  

Besides sliding cycles, crossing limit cycles can exist for PWS systems and for example J. Llibre and co-workers have obtained upper bounds for a number of classes  \cite{esteban2021a,li2021a,llibre2013a}. See also \cite{Freire,carmona2023a,Llibre3LC,LlibreOrd,BragaMello,HuanYang,Han2010,Gasull2020} and references therein.

\smallskip

Section \ref{sectionVI3} is devoted to applications of the slow divergence integral from Section \ref{section2general}. In Section \ref{sectionVI3} we focus on the model used in \cite{RHKK} (visible-invisible two-fold \textbf{$VI_3$}) and read upper bounds of the number of sliding limit cycles and type of bifurcations near so-called generalized canard cycles (Fig. \ref{fig-LPS}) from fractal properties of a bounded and monotone sequence in $\mathbb R$ defined using the slow divergence integral and the notion of slow relation function (see Section \ref{model+assumptions}). The main advantage of this fractal approach is that, instead of computing the multiplicity of zeros of the slow divergence integral (like in \cite{RHKK}), it suffices to find the Minkowski dimension \cite{Falconer} of the sequence. There is a bijective correspondence between the multiplicity and the Minkowski dimension (see Section \ref{section-proofs}). A similar fractal approach has been used in \cite{BoxRenato,BoxDomagoj,BoxVlatko,BoxNovo} where one deals with smooth slow-fast systems. See also \cite{EZZ,zubzup05} and the references therein. We point out that there exist simple formulas for numerical computation of the Minkowski dimension of the sequence (see e.g. \cite{BoxVlatko}). In Section \ref{section-statement-fractal} we state the main fractal results (Theorems \ref{theorem-1}--\ref{theorem-3}), and in Section \ref{section-proofs} we prove them.
\smallskip

For the sake of readability, in this paper we work in $\mathbb R^2$ using the Euclidean metric. We believe that the notion of slow divergence integral in regularized PWS systems on a smooth surface could also be defined. We point out that the slow divergence integral \cite{DDR-book-SF} is defined for slow-fast systems on a smooth surface.

\smallskip

\section{The slow divergence integral in PWS systems with sliding}\label{section2general}

First we recall the basic definitions in PWS theory \cite{Bernardo08,GST2011} (switching boundary, sliding set, crossing set, sliding vector field, tangency point, two-fold singularity, etc.). Then we define the notion of slow divergence integral of a regularized PWS system along a sliding segment (not containing singularities of the sliding vector field) and prove that the integral is invariant under smooth equivalences (see Section \ref{sectionSDI-invariance}). In Section \ref{section-SDI2Fold} we extend the definition of the slow divergence integral to segments consisting of a stable sliding region, an unstable sliding region and a two-fold singularity between them. If the two-fold singularity is visible-invisible, then we assume that the sliding vector field is regular or has a hyperbolic singularity in the two-fold point. In Section \ref{section-SDI1Tang} we define the slow divergence integral near a tangency point where the tangency (quadratic or more degenerate) appears only in one vector field. 

\smallskip

Consider a $\lambda$-family of PWS systems in the plane
\begin{align}
 \dot z = \begin{cases}
           Z_\lambda^+(z) &\text{for}\quad z\in\Sigma_\lambda^+,\\
           Z_\lambda^-(z) &\text{for}\quad z\in\Sigma_\lambda^-,
          \end{cases}\label{zZpminvariance}
\end{align}
where $z=(x,y)$, $\lambda\sim\lambda_0\in\mathbb R^l$ and the switching boundary is a smooth $\lambda$-family of $1$-dimensional manifolds $\Sigma_\lambda$ given by
\[\Sigma_\lambda=\{z\in\mathbb R^2 \ | \ h_\lambda(z)=0\}\cap V,\]
with an open set $V$ and a smooth family of functions $h_\lambda$ such that $\nabla h_\lambda(z)\ne (0,0)$, $\forall z\in \Sigma_\lambda$. The switching boundary $\Sigma_\lambda$ separates the open set $\Sigma_\lambda^+=\{z\in V \ | \ h_\lambda(z)>0\}$ from the open set $\Sigma_\lambda^-=\{z\in V \ | \ h_\lambda(z)<0\}$. We assume that the $\lambda$-family of vector fields $Z_\lambda^+=(X_\lambda^+,Y_\lambda^+)$ (resp. $Z_\lambda^-=(X_\lambda^-,Y_\lambda^-)$) is smooth in the closure of the $\lambda$-family $\Sigma_\lambda^+$ (resp. $\Sigma_\lambda^-$). In this paper ``smooth'' means ``$C^\infty$-smooth''.
\smallskip

The subset $\Sigma_\lambda^{sl}\subset \Sigma_\lambda$ consisting of all points $z\in\Sigma_\lambda$ such that $$Z_\lambda^+(h_\lambda)(z)Z_\lambda^-(h_\lambda)(z)<0$$ is called the sliding set, where $Z_\lambda^\pm(h_\lambda)(z):=\nabla h_\lambda(z)\cdot Z_\lambda^\pm(z)$ is the Lie-derivative of $h_\lambda$ with respect to the vector field $Z_\lambda^\pm$ at $z$. A sliding point $z\in\Sigma_\lambda^{sl}$ is stable (resp. unstable) if $Z_\lambda^+(h_\lambda)(z)<0$ and $Z_\lambda^-(h_\lambda)(z)>0$ (resp. $Z_\lambda^+(h_\lambda)(z)>0$ and $Z_\lambda^-(h_\lambda)(z)<0$). We write $\Sigma_\lambda^{sl}=\Sigma_\lambda^{s}\cup\Sigma_\lambda^{u}$ where $\Sigma_\lambda^{s}$ (resp. $\Sigma_\lambda^{u}$) is the set of all stable (resp. unstable) sliding points. In $\Sigma_\lambda^{s}$ (resp. $\Sigma_\lambda^{u}$) the vector fields $Z_\lambda^\pm$ point toward (resp. away from) the switching boundary. We call the set $\Sigma_\lambda^{cr}\subset \Sigma_\lambda$ of all points $z\in\Sigma_\lambda$ such that $$Z_\lambda^+(h_\lambda)(z)Z_\lambda^-(h_\lambda)(z)>0$$ the crossing set.
\smallskip

At each point $z\in\Sigma_\lambda^{cr}$ the orbit of \eqref{zZpminvariance} crosses the switching boundary $\Sigma_\lambda$ (it is the concatenation of the orbit of $Z_\lambda^+$ and the orbit of $Z_\lambda^-$ through $z$). Along the sliding set $\Sigma_\lambda^{sl}$, the flow is given by the Filippov sliding vector field \cite{filippov1988differential}
\begin{equation}\label{FilippovSVFGeneral}
    Z_\lambda^{sl}(z)=\frac{1}{(Z_\lambda^+-Z_\lambda^-)(h_\lambda)}\left(Z_\lambda^+(h_\lambda)Z_\lambda^--Z_\lambda^-(h_\lambda)Z_\lambda^+\right)(z), \ z\in \Sigma_\lambda^{sl}.
\end{equation}
The sliding vector field $Z_\lambda^{sl}$ defined in \eqref{FilippovSVFGeneral} is tangent to $\Sigma_\lambda^{sl}$, i.e., $Z_\lambda^{sl}(z)$ is equal to the convex combination $\tau Z_\lambda^+(z)+(1-\tau)Z_\lambda^-(z)$ with
\begin{equation}\label{taulambda}
    \tau=\tau_\lambda(z)=\frac{-Z_\lambda^-(h_\lambda)}{(Z_\lambda^+-Z_\lambda^-)(h_\lambda)}(z)\in ]0,1[.
\end{equation}
We say that $z\in\Sigma_\lambda^{sl}$ is a pseudo-equilibrium of \eqref{zZpminvariance} if $Z_\lambda^{sl}(z)=0$.
\smallskip

A point $z\in\Sigma_\lambda$ where $Z_\lambda^+(h_\lambda)(z)=0$ or $Z_\lambda^-(h_\lambda)(z)=0$ is a PWS singularity called  tangency. We say that $z\in\Sigma_\lambda$ is a fold singularity (or a fold point) of $Z_\lambda^+$ (resp. $Z_\lambda^-$) if $Z_\lambda^+(h_\lambda)(z)=0$ and $(Z_{\lambda}^+)^2(h_{\lambda})(z)\ne 0$ (resp. $Z_\lambda^-(h_\lambda)(z)=0$ and $(Z_{\lambda}^-)^2(h_{\lambda})(z)\ne 0$). The fold point is visible if $(Z_{\lambda}^+)^2(h_{\lambda})(z)> 0$ (resp. $(Z_{\lambda}^-)^2(h_{\lambda})(z)<0$) and invisible if $(Z_{\lambda}^+)^2(h_{\lambda})(z)< 0$ (resp. $(Z_{\lambda}^-)^2(h_{\lambda})(z)> 0$). 
\smallskip

We say that $z\in\Sigma_\lambda$ is a two-fold singularity if $z$ is a fold point of both $Z_\lambda^\pm$. A two-fold singularity $z\in\Sigma_\lambda$ is said to be visible-visible ($VV$) if $z$ is visible in both $Z_\lambda^\pm$, invisible-invisible ($II$) if $z$ is invisible in both $Z_\lambda^\pm$, and visible-invisible ($VI$) if $z$ is visible in $Z_\lambda^+$ and invisible in $Z_\lambda^-$ or visible in $Z_\lambda^-$ and invisible in $Z_\lambda^+$. Following \cite{Kuznetsov2003}, there exist $7$ (generic) cases for two-fold singularities taking into account the direction of the flow of $Z_\lambda^\pm$ and $Z_\lambda^{sl}$: $2$ visible-visible cases (denoted by $VV_{1}$ and $VV_2$ in \cite{Kuznetsov2003}), $2$ invisible-invisible cases ($II_{1}$ and $II_2$) and $3$ visible-invisible cases ($VI_{1}$, $VI_2$ and $VI_3$). For more details we refer to \cite{bonet-reves2018a,GST2011,kristiansen2015a,Kuznetsov2003}.
In Section \ref{section-SDI2Fold} we define the notion of slow divergence integral near two-fold singularities of sliding type ($VV_1$, $II_1$, $VI_2$ and $VI_3$). The four sliding cases are illustrated in Fig. \ref{fig-sliding2Fold}. We also treat a visible-invisible two-fold singularity where the sliding vector field points toward (or away from) the two-fold singularity on both sides (Fig. \ref{fig-sliding2FoldNonGeneric}).

\smallskip

We consider a regularized PWS system \cite{RHKK}
\begin{align}\label{regul-gener}
  \dot z &=\phi(h_\lambda(z)\epsilon^{-1}) Z_\lambda^+(z) + (1-\phi(h_\lambda(z)\epsilon^{-1}))Z_\lambda^-(z)
 \end{align}
where $0<\epsilon\ll 1$ and $\phi:\mathbb{R}\to \mathbb{R}$ is a smooth regularization function. 
 We assume that $\phi$ is strictly monotone, i.e., 
 \begin{equation}
     \label{strictly-mono}
     \phi'(u)>0 \text{ for all } u\in \mathbb R,
 \end{equation}
  and $\phi$ has the following asymptotics for $u\to\pm\infty$: \begin{align}\label{phi-asim}
 \phi(u)\rightarrow \begin{cases}
                     1 & \text{for}\quad  u\rightarrow \infty,\\
                     0 & \text{for}\quad u\rightarrow -\infty.
                    \end{cases}
                    \end{align}
 Moreover, we assume that $\phi$ is smooth at $\pm \infty$ in the following sense: The functions
\begin{align*}
 \phi_+(u):=\begin{cases}
             1 & \text{for}\quad u=0,\\
             \phi(u^{-1}) & \text{for}\quad u>0,
            \end{cases},\quad
            \phi_-(u):=\begin{cases}
             \phi(-u^{-1}) & \text{for}\quad u>0,\\
             0 & \text{for}\quad u=0,
            \end{cases}
\end{align*}
are smooth at $u=0$.

Using the asymptotics of $\phi$ given in \eqref{phi-asim}, the system \eqref{regul-gener} becomes the PWS system \eqref{zZpminvariance}  in the limit $\epsilon\rightarrow 0$. Combining this with the fact that $\phi$ is smooth at $\pm \infty$, we have that, for $z$ kept in any 
fixed compact set in $V\setminus \Sigma_\lambda$, the right hand side of \eqref{regul-gener} is an $o(1)$-perturbation of the right hand side of \eqref{zZpminvariance} where $o(1)$ is a smooth function in $(z,\epsilon,\lambda)$ and tends to $0$ as $\epsilon\to 0$, uniformly in $(z,\lambda)$. Thus, the PWS system \eqref{zZpminvariance} describes the dynamics of \eqref{regul-gener}, for $\epsilon>0$ small, as long as $z$ is kept uniformly away from $\Sigma_\lambda$. 
\smallskip

Near $\Sigma_\lambda^{sl}$, the dynamics of \eqref{regul-gener}, with $0<\epsilon\ll 1$, is given by Proposition \ref{prop-sliding} (see also \cite{Sotomayor96}).

\subsection{Definition and invariance of the slow divergence integral}\label{sectionSDI-invariance} 

\begin{proposition}\label{prop-sliding}
    Suppose that the PWS system \eqref{zZpminvariance} has a stable (resp. unstable) sliding point $p\in\Sigma_{\lambda_0}^{sl}$. Then, for each $0<\epsilon\ll 1$ and $\lambda\sim\lambda_0$, \eqref{regul-gener} has a locally invariant manifold near $p$ with foliation by stable (resp. unstable) fibers, and the reduced dynamics on this manifold (when $\epsilon\to 0$) is given by sliding vector
field $Z_\lambda^{sl}$ defined in \eqref{FilippovSVFGeneral}.
\end{proposition}
\begin{proof}
Without loss of generality, we can assume that $\frac{\partial h_{\lambda_0}}{\partial y}(p)\ne 0$. Then the switching boundary $\Sigma_\lambda$ (locally near $p$) is the graph of a smooth function $y=f_\lambda(x)$. Using $h_\lambda(x,y)=\epsilon \tilde y$, the system \eqref{regul-gener} multiplied by $\epsilon>0$ becomes a slow-fast system
\begin{equation}\label{rescaling-equation}
\begin{aligned}
 \dot x &=\epsilon\left(\phi(\tilde y)X_\lambda^+(x,f_\lambda(x))+(1-\phi(\tilde y))X_\lambda^-(x,f_\lambda(x))+O(\epsilon\tilde y)\right),\\
 \dot {\tilde y} &=\phi(\tilde y)Z_\lambda^+(h_\lambda)(x,f_\lambda(x))+(1-\phi(\tilde y))Z_\lambda^-(h_\lambda)(x,f_\lambda(x))+O(\epsilon\tilde y).
\end{aligned}
 \end{equation}
When $\epsilon=0$, the curve of singularities of \eqref{rescaling-equation} is given by $\tilde y=\phi^{-1}\left(\tau_\lambda(x,f_\lambda(x))\right)$ where $\tau_\lambda$ is defined in \eqref{taulambda}. Each singularity $(x,\phi^{-1}\left(\tau_\lambda(x,f_\lambda(x))\right))$ is semi-hyperbolic with the nonzero eigenvalue equal to the divergence of the vector field \eqref{rescaling-equation}, with $\epsilon=0$, computed in that singularity:
\begin{equation}\label{divmotivation}
    (Z_\lambda^+-Z_\lambda^-)(h_\lambda)(x,f_\lambda(x))\phi'\left(\phi^{-1}\left(\tau_\lambda(x,f_\lambda(x))\right)\right).
\end{equation}
The reason why the eigenvalue in \eqref{divmotivation} is nonzero is because $Z_\lambda^+(h_\lambda)Z_\lambda^-(h_\lambda)<0$ and $\phi'>0$ (see \eqref{strictly-mono}). The curve of singularities is attracting (resp. repelling) if $p$ is a stable (resp. unstable) sliding point. The result follows now from Fenichel’s theory \cite{fen3}. Notice that the reduced dynamics of \eqref{rescaling-equation} along the curve of singularities is given by the vector field 
\begin{equation}\label{resc-SVF}\left(\tau_\lambda X_\lambda^++(1-\tau_\lambda)X_\lambda^-\right)(x,f_\lambda(x)).
\end{equation}
We divided the $x$-component in \eqref{rescaling-equation} by $\epsilon$ and let $\epsilon\to 0$ with $\tilde y=\phi^{-1}\left(\tau_\lambda(x,f_\lambda(x))\right)$. We get the same expression \eqref{resc-SVF} if we use the definition of the slow vector field introduced in Section \ref{Section-Introduction}.
This completes the proof.
\end{proof}

Following \cite[Chapter 5]{DDR-book-SF} or Section \ref{Section-Introduction} in the smooth slow-fast system \eqref{rescaling-equation} one can define the notion of slow divergence integral along normally hyperbolic curve of singularities $\tilde y=\phi^{-1}\left(\tau_\lambda(x,f_\lambda(x))\right)$ when the sliding vector field in \eqref{resc-SVF} has no singularities: it is the integral of the divergence in \eqref{divmotivation} where the variable of integration is the time variable of the flow of the sliding vector field. This is our motivation for the definition of the notion of slow divergence integral of regularized PWS system \eqref{regul-gener} (see also \cite{RHKK}).

\begin{definition}[\textbf{Slow divergence integral}] \label{definition-SDIGen}
 Let $m_\lambda\subset \Sigma_\lambda^{sl}$  be a bounded segment (Fig. \ref{fig-slidsegment}) not containing pseudo-equilibria of the PWS system \eqref{zZpminvariance}. Let $z_\lambda:[ t_1, t_2]\to \mathbb R^2$ be a solution of $z'( t)=Z_\lambda^{sl}(z( t))$ where $z_\lambda( t_1)$ and $z_\lambda( t_2)$ are the end points of $m_\lambda$ ($z_\lambda$ is a parameterization of $m_\lambda$). Then we define the slow divergence integral of regularized PWS system \eqref{regul-gener} associated to $m_\lambda$ as 
 \begin{equation}\label{slow-divergence-int}
I(m_\lambda)=\int_{t_1}^{ t_2}E_\lambda(z_\lambda( t))d t
 \end{equation}
 where 
 $$E_\lambda(z)=(Z_\lambda^+-Z_\lambda^-)(h_\lambda)(z)\phi'\left(\phi^{-1}\left(\tau_\lambda(z)\right)\right), \ z\in \Sigma_\lambda^{sl}.$$
\end{definition}
\begin{remark}
    Note that the definition of the slow divergence integral given by \eqref{slow-divergence-int} is independent of the choice of $z_\lambda$. Indeed, if $\hat z_\lambda$ is another solution to $z'( t)=Z_\lambda^{sl}(z( t))$ and $p\in m_\lambda$, then there exist $\tilde t\in [t_1,t_2]$ and $\bar t$ such that $z_\lambda(\tilde t)=\hat z_\lambda(\bar t)=p$. Then we have $z_\lambda( t)=\hat z_\lambda(t+\bar t-\tilde t)$ due to uniqueness of solutions. Now, we get
    $$\int_{t_1+\bar t-\tilde t}^{ t_2+\bar t-\tilde t}E_\lambda(\hat z_\lambda(s))d s=\int_{t_1}^{ t_2}E_\lambda(z_\lambda( t))d t,$$
where we use the change of variable $s=t+\bar t-\tilde t$.  
\end{remark}
 If $m_\lambda$ is stable (resp. unstable), then $I(m_\lambda)$ is negative (resp. positive).

\smallskip

The slow divergence integral from Definition \ref{definition-SDIGen} is invariant under smooth equivalences (Theorem \ref{theorem-invariance}.1 and Theorem \ref{theorem-invariance}.2). Theorem \ref{theorem-invariance}.3 tells us how to compute $I(m_\lambda)$ for an arbitrary  parameterization of $m_\lambda$ (see also \cite[Proposition 5.3]{DDR-book-SF}).

\begin{figure}[htb]
	\begin{center}
		\includegraphics[width=4.7cm,height=4.4cm]{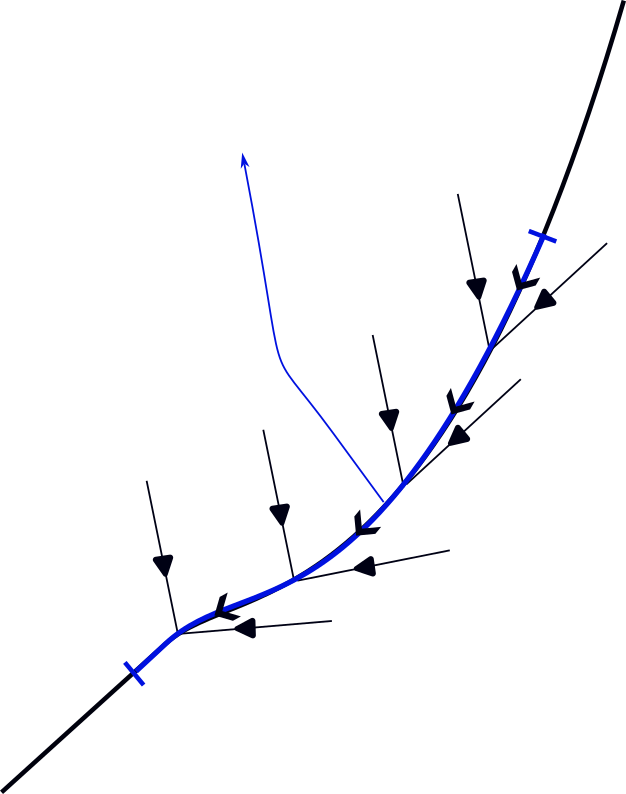}
		{\footnotesize
         \put(-87,105){$m_\lambda$}
        \put(-15,130){$\Sigma_\lambda^{sl}$}
}
         \end{center}
	\caption{A segment $m_\lambda\subset \Sigma_\lambda^{sl}$ (blue).}
	\label{fig-slidsegment}
\end{figure}

\begin{theorem}[\textbf{Invariance of the slow divergence integral}]\label{theorem-invariance} Let us denote the family of vector field in \eqref{regul-gener} by $Z_{\epsilon,\lambda}$ and let $m_\lambda\subset \Sigma_\lambda^{sl}$ be as in Definition \ref{definition-SDIGen}. The following statements are true.
\begin{enumerate}
    \item Let $T:V_w\to V_z\subset V$ ($w\mapsto z=T(w)$) be a smooth coordinate transformation, with open sets $V_w,V_z\subset \mathbb R^2$. Let $I(m_\lambda)$ be the slow divergence integral of $Z_{\epsilon,\lambda}$ along $m_\lambda\subset V_z$. Then the slow divergence integral $I(T^{-1}(m_\lambda))$ of the pullback of the vector field $Z_{\epsilon,\lambda}|_{V_z}$ along $T^{-1}(m_\lambda)\subset V_w$ is equal to $I(m_\lambda)$.
    \item Let $g$ be a smooth strictly positive function defined in a neighborhood of $m_\lambda$.  Then the slow divergence integral of $Z_{\epsilon,\lambda}$ along $m_\lambda$ is equal to the slow divergence integral of the equivalent
vector field $g.Z_{\epsilon,\lambda}$ along $m_\lambda$.
\item Let $p_\lambda:[v_1,v_2]\to \mathbb R^2$ be a parameterization of $m_\lambda$. Then we have $$I(m_\lambda)=\int_{ v_1}^{v_2}\frac{E_\lambda(p_\lambda(v))d v}{|\tilde p_\lambda(v)|},$$ 
where $\tilde p_\lambda$ is a smooth $\lambda$-family of nowhere zero functions satisfying $$Z_\lambda^{sl}(p_\lambda(v))=\tilde p_\lambda(v)p_\lambda'(v).$$
\end{enumerate}
    \end{theorem}
\begin{proof}
\emph{Statement 1}. The pullback of  the vector field $Z_{\epsilon,\lambda}|_{V_z}$ can be written as 
\begin{align*}
    T^*(Z_{\epsilon,\lambda}|_{V_z})(w)=&\phi(h_\lambda\circ T(w)\epsilon^{-1}) W_\lambda^+(w)+ (1-\phi(h_\lambda\circ T(w)\epsilon^{-1}))W_\lambda^-(w)
    \end{align*}
    where $W_\lambda^\pm(w)=DT(w)^{-1}(Z_\lambda^\pm\circ T)(w)$. It is not difficult to see that the Lie-derivative of $h_\lambda\circ T$ with respect to the vector field $W_\lambda^\pm$ is given by
    \begin{equation}\label{transGeneral}
   W_\lambda^\pm(h_\lambda\circ T)(w)=Z_\lambda^\pm(h_\lambda)(T(w)).
    \end{equation}
    Using \eqref{transGeneral} and Definition \ref{definition-SDIGen} we find that the slow divergence integral of $T^*(Z_{\epsilon,\lambda}|_{V_z})$ along $T^{-1}(m_\lambda)$ is given by 
$$I(T^{-1}(m_\lambda))=\int_{t_1}^{ t_2}E_\lambda(T(w_\lambda( t)))d t$$
where $w_\lambda:[ t_1, t_2]\to  T^{-1}(m_\lambda)$ is a solution of $w'( t)=W_\lambda^{sl}(w( t))$ (the Filippov sliding vector field along $ T^{-1}(m_\lambda)$ is given by $W_\lambda^{sl}(w)=DT(w)^{-1}Z_\lambda^{sl}(T(w))$). Since $T\circ w_\lambda:[t_1, t_2]\to m_\lambda$ is a solution to $z'( t)=Z_\lambda^{sl}(z( t))$, the result follows.\\

\emph{Statement 2}. From Definition \ref{definition-SDIGen} it follows that the slow divergence integral of $g.Z_{\epsilon,\lambda}$ along $m_\lambda$ is equal to
\begin{equation}\label{int-equiv}
    I(m_\lambda)=\int_{\hat t_1}^{\hat t_2}E_\lambda(\hat z_\lambda(\hat t))g(\hat z_\lambda(\hat t))d\hat t
\end{equation}
where $\hat z_\lambda:[\hat t_1,\hat t_2]\to  m_\lambda$ and $\hat z_\lambda'(\hat t)=g(\hat z_\lambda(\hat t))Z_\lambda^{sl}(\hat z_\lambda(\hat t))$. We make in the integral in \eqref{int-equiv} the change of variable $ t=\rho(\hat t)=\int_{\hat t_1}^{\hat t}g(\hat z_\lambda( v))d v$ with $\hat t\in [\hat t_1,\hat t_2]$. Then we have
$$\int_{\hat t_1}^{\hat t_2}E_\lambda(\hat z_\lambda(\hat t))g(\hat z_\lambda(\hat t))d\hat t=\int_{0}^{\rho(\hat t_2)}E_\lambda(\hat z_\lambda\circ \rho^{-1}( t))d t.$$
Since $(\hat z_\lambda\circ\rho^{-1})'( t)=Z_\lambda^{sl}((\hat z_\lambda\circ\rho^{-1})( t))$, $ t\in [0,\rho(\hat t_2)]$, this integral is the slow divergence integral of $Z_{\epsilon,\lambda}$ associated to $m_\lambda$. This completes the proof of Statement 2.

\emph{Statement 3}. The proof of Statement 3 is similar to the proof of Statement 2. 
\end{proof}

\begin{remark}
    \label{invariance-changetime}
    It follows directly from Definition \ref{definition-SDIGen} that the slow divergence integral of $-Z_{\epsilon,\lambda}$ along $m_\lambda$ is equal to the slow divergence integral of $Z_{\epsilon,\lambda}$ along $m_\lambda$ multiplied by $-1$.
\end{remark}

We will use the invariance of the slow divergence integral under smooth
equivalences in Section \ref{section-SDI2Fold} and Section \ref{section-SDI1Tang}.

If $m_\lambda\subset \Sigma_\lambda^{sl}$ contains pseudo-equilibria, then the slow divergence integral associated to $m_\lambda$ is not well-defined.

\subsection{The slow divergence integral near two-fold singularities}\label{section-SDI2Fold}

In this section we suppose that the sliding vector field $Z_\lambda^{sl}$, given by \eqref{FilippovSVFGeneral}, is defined around a two-fold singularity. Our goal is to define the notion of slow divergence integral near such a two-fold singularity. Since the slow divergence integral \eqref{slow-divergence-int} is invariant under smooth equivalences (Theorem \ref{theorem-invariance}), we use a normal form of \eqref{zZpminvariance}, locally near the two-fold singularity, in which $h_\lambda(x,y)=y$ and the two-fold point corresponds to the origin $p=(0,0)$. Notice that such normal form coordinates exist because $\nabla h_\lambda(z)\ne (0,0)$, $\forall z\in \Sigma_\lambda$, in \eqref{zZpminvariance}.
\smallskip

Using $h_\lambda(x,y)=y$ the two-fold $p$ satisfies
\begin{equation}\label{conditionsTwoFold}
Z_\lambda^\pm(h_\lambda)(0)=Y_\lambda^\pm(0)=0, \ \ (Z_{\lambda}^\pm)^2(h_{\lambda})(0)=X_\lambda^\pm(0)\partial_xY_\lambda^\pm(0)\ne 0,
\end{equation}
and the sliding vector field $Z_\lambda^{sl}$ near $p$ can be written as
\begin{equation}\label{FSVF}
X_\lambda^{sl}(x)=\frac{\det Z_\lambda(x)}{(Y_\lambda^--Y_\lambda^+)(x,0)}
\end{equation}
where $$\det Z_\lambda(x):=(X_\lambda^+Y_\lambda^--X_\lambda^-Y_\lambda^+)(x,0).$$
\begin{remark}
    The notation $\det Z_\lambda$ comes from \cite{bonet-reves2018a}. In \cite{RHKK} a similar notation has been used for $-(X_\lambda^+Y_\lambda^--X_\lambda^-Y_\lambda^+)$. 
\end{remark}
Since we assumed that the sliding vector field $X_\lambda^{sl}$ is defined around the two-fold $p$, we find that $X_\lambda^+(0)X_\lambda^-(0)>0$ if the folds have the same visibility (visible-visible or invisible-invisible) and $X_\lambda^+(0)X_\lambda^-(0)<0$ if the folds have opposite visibility. We have $p\in\partial\Sigma_\lambda^{s}\cap\partial\Sigma_\lambda^{u}$. These properties follow directly from \eqref{conditionsTwoFold} and the definition of visible and invisible folds (see \cite[Lemma 2.8]{bonet-reves2018a}), and imply that $\partial_x (Y_\lambda^-- Y_\lambda^+)(0)\ne 0$ and $\partial_xY_\lambda^+(0)\partial_xY_\lambda^-(0)<0$. 
\smallskip

Using $\partial_x (Y_\lambda^-- Y_\lambda^+)(0)\ne 0$ it is clear that the sliding vector field in \eqref{FSVF} has a removable singularity in $x=0$ and 
\begin{equation}
    \label{removablesing}
    X_\lambda^{sl}(x)=\nu+O(x), \ \ \nu=\frac{(\det Z_\lambda)'(0)}{\partial_x (Y_\lambda^-- Y_\lambda^+)(0)}.
\end{equation}
From \cite[Lemma 2.9]{bonet-reves2018a} and \cite[Corollary 2.10]{bonet-reves2018a} it follows that $\nu\ne 0$ and $\sgn(\nu)=\sgn(X_\lambda^+(0))$ if the folds have the same visibility ($VV_1$ and $II_1$ in Fig. \ref{fig-sliding2Fold}), and that $\nu\ne 0$ and $\sgn(\nu)=-\sgn\left(X_\lambda^+(0)(\det Z_\lambda)'(0)\right)$ if the folds have opposite visibility and $(\det Z_\lambda)'(0)\ne 0$ ($VI_2$ and $VI_3$ in Fig. \ref{fig-sliding2Fold}). \textit{If the folds have opposite visibility, we assume that $\nu\ne 0$ in \eqref{removablesing} or $x=0$ is a hyperbolic singularity of the sliding vector field $X_\lambda^{sl}$ (or, equivalently, $x=0$ is a zero of multiplicity $1$ or $2$ of the function $\det Z_\lambda$).} We refer to Fig. \ref{fig-sliding2FoldNonGeneric} (the multiplicity of the zero $x=0$ of $\det Z_\lambda$ is $2$). 

\begin{figure}[htb]
	\begin{center}
		\includegraphics[width=8.9cm,height=5.4cm]{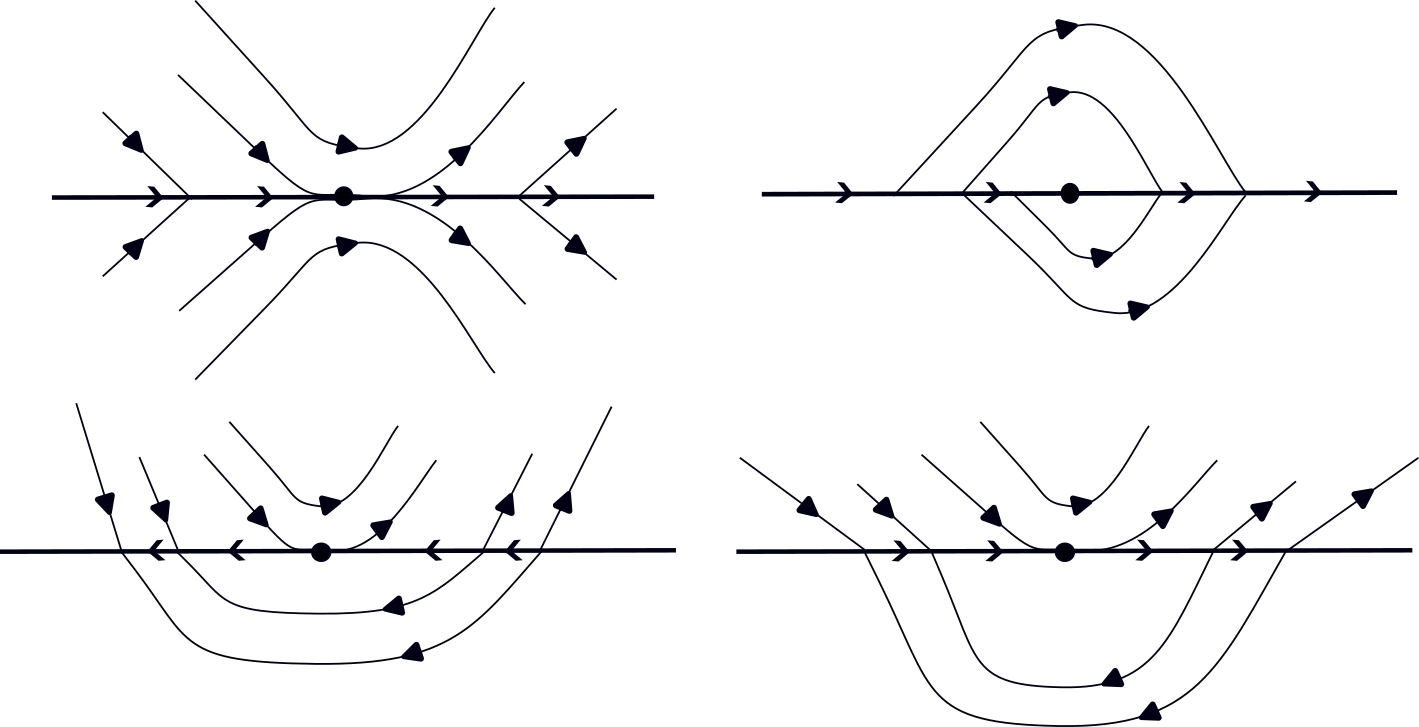}
		{\footnotesize
        \put(-125,62){$VI_3$}
        \put(-260,62){$VI_2$}
        \put(-125,140){$II_1$}
        \put(-260,140){$VV_1$}
}
         \end{center}
	\caption{The different types of two-fold
singularities with sliding: the folds in $VV_1$ and $II_1$ have the same visibility, while the folds in $VI_2$ and $VI_3$ have opposite visibility.}
	\label{fig-sliding2Fold}
\end{figure}
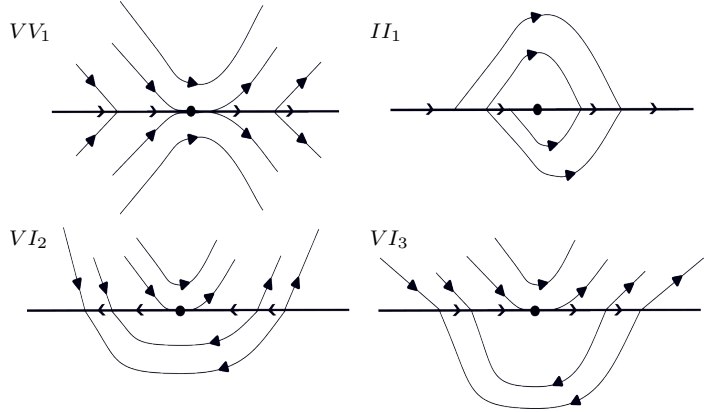

From $\partial_x (Y_\lambda^-- Y_\lambda^+)(0)\ne 0$ and $\partial_xY_\lambda^+(0)\partial_xY_\lambda^-(0)<0$ it follows that the function $\tau_\lambda$ defined in \eqref{taulambda} has the following property when $x\to 0$:
\begin{equation}
    \label{tau-extended}
    \lim_{x\to 0}\tau_\lambda(x,0)=\lim_{x\to 0} \frac{-Y_\lambda^-}{Y_\lambda^+-Y_\lambda^-}(x,0)=\frac{\partial_xY_\lambda^-(0)}{\partial_x (Y_\lambda^-- Y_\lambda^+)(0)}\in ]0,1[.
\end{equation}
\smallskip

Let us now compute the slow divergence integral along $[x_0,x_1]$, with $0<x_0<x_1$. We assume that $x_1$ is small enough such that $[x_0,x_1]$ does not contain any singularities of the sliding vector field $X_\lambda^{sl}$. We use Theorem \ref{theorem-invariance}.3. We take $p_\lambda(x)=(x,0)$, $x\in [x_0,x_1]$, in Theorem \ref{theorem-invariance}.3. Then we have 
$$\tilde p_\lambda(x)=\frac{\det Z_\lambda(x)}{(Y_\lambda^--Y_\lambda^+)(x,0)}, \  E_\lambda(p_\lambda(v))=(Y_\lambda^+-Y_\lambda^-)(x,0)\phi'\left(\phi^{-1}\left(\tau_\lambda(x,0)\right)\right).$$
This implies 
\begin{equation}
    \label{SDIPomoc}
I([x_0,x_1])=\int_{x_0}^{x_1}\frac{|Y_\lambda^--Y_\lambda^+|(Y_\lambda^+-Y_\lambda^-)(x,0)}{|\det Z_\lambda|(x)}\phi'\left(\phi^{-1}\left(\frac{-Y_\lambda^-}{Y_\lambda^+-Y_\lambda^-}(x,0)\right)\right)dx.
\end{equation}
\begin{figure}[htb]
	\begin{center}
		\includegraphics[width=9.4cm,height=2.5cm]{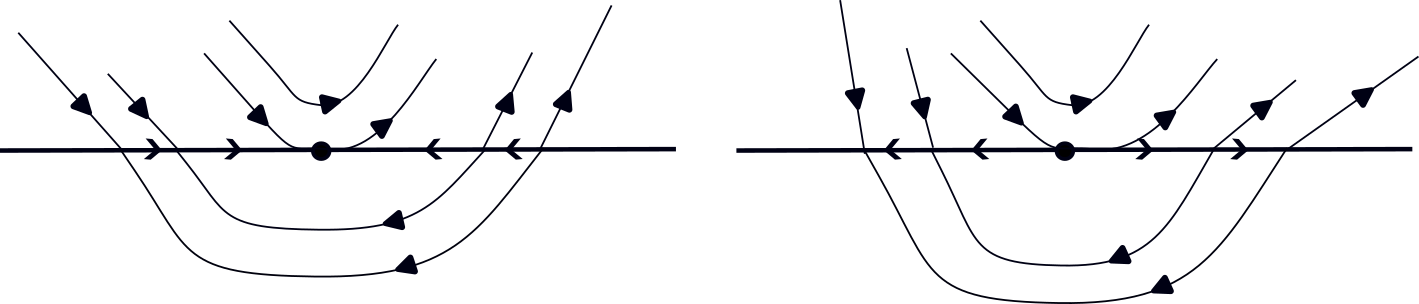}
		{\footnotesize
        \put(-75,-9){$(b)$}
        \put(-213,-9){$(a)$}
}
         \end{center} 
	\caption{Non-generic visible-invisible two-fold singularities. The (extended) sliding vector field has a hyperbolic singularity at the two-fold points. (a) The sliding vector field points toward the two-fold singularity. (b) The sliding vector field points away from the two-fold singularity.}
	\label{fig-sliding2FoldNonGeneric}
\end{figure}
Finally, we define the slow divergence integral along $[0,x_1]$ (the left end point of $[0,x_1]$ is the two-fold point). 
\begin{definition}\label{def-SDIContact} Let $m_\lambda=[0,x_1]$. Then the slow divergence integral along $m_\lambda$ is defined as 
   \begin{equation}
       I(m_\lambda)=\lim_{x_0\to 0+}I([x_0,x_1]) \nonumber
       \end{equation}
       where $I([x_0,x_1])$ is given in \eqref{SDIPomoc}.
\end{definition}
\begin{remark} Notice that the function $x\mapsto \phi'\left(\phi^{-1}\left(\tau_\lambda(x,0)\right)\right)$ in \eqref{SDIPomoc} can be defined at $x=0$ such that this function is smooth and positive on the segment $m_\lambda$ (see \eqref{strictly-mono}, \eqref{phi-asim} and \eqref{tau-extended}).
If the folds have the same visibility, then $I(m_\lambda)$ is well-defined (finite) because $\nu\ne 0$ in \eqref{removablesing}. 
    Since we assume that the multiplicity of the zero $x=0$ of $\det Z_\lambda$ does not exceed $2$ when the folds have opposite visibility, $I(m_\lambda)$ is finite.
\end{remark}
\begin{remark}\label{remark-SDI-negative} 
   The slow divergence integral along $m_\lambda=[ x_0,0]$, with $x_0<0$, can be defined in a similar way: $I(m_\lambda)=\lim_{x_1\to 0-}I([x_0,x_1])$
       where $I([x_0,x_1])$ has the same form \eqref{SDIPomoc}.
\end{remark}

\subsection{The slow divergence integral near one-sided tangency points}\label{section-SDI1Tang}
In this section we define the slow divergence integral near a tangency point $p\in\partial\Sigma_\lambda^{s}\cup\partial\Sigma_\lambda^{u}$ where both vectors $Z_\lambda^\pm(p)$ are nonzero
and precisely one of them is tangent to $\Sigma_\lambda$ at $p$ (see e.g. Fig. \ref{fig-OneTangency}). Like in Section \ref{section-SDI2Fold}, the switching boundary $\Sigma_\lambda$ is locally given by $h_\lambda(x,y)=y$ and $p=(0,0)$. Since we suppose that $p\in\partial\Sigma_\lambda^{s}\cup\partial\Sigma_\lambda^{u}$, there is a side of $p$ (without loss of generality we take $x>0$) where the sliding vector field is defined, and given by \eqref{FSVF}. If the (nonzero) vector $Z_\lambda^+(0)$ (resp. $Z_\lambda^-(0)$) is  tangent to $\Sigma_\lambda$, then $X_\lambda^+(0)\ne 0$, $Y_\lambda^+(0)=0$ and $Y_\lambda^-(0)\ne 0$ (resp. $X_\lambda^-(0)\ne 0$, $Y_\lambda^-(0)=0$ and $Y_\lambda^+(0)\ne 0$) and 
\begin{equation}
    \label{removablesingOneTang}
    X_\lambda^{sl}(x)=X_\lambda^+(0)+O(x) \  \left(\text{resp. } X_\lambda^{sl}(x)=X_\lambda^-(0)+O(x) \right).
\end{equation}
Since $X_\lambda^+(0)\ne 0$ (resp. $X_\lambda^-(0)\ne 0$), the sliding vector field $X_\lambda^{sl}$ in \eqref{removablesingOneTang} is regular near $x=0$. Thus, the segment $[x_0,x_1]$, with $0<x_0<x_1$, does not contain any singularities of $X_\lambda^{sl}$ if $x_1$ is small enough and we can define the slow divergence integral along $[x_0,x_1]$ exactly in the same way as in Section \ref{section-SDI2Fold}. The slow divergence integral is given by \eqref{SDIPomoc} and we use the same notation $I([x_0,x_1])$.

\begin{figure}[htb]
	\begin{center}
		\includegraphics[width=10.4cm,height=2.5cm]{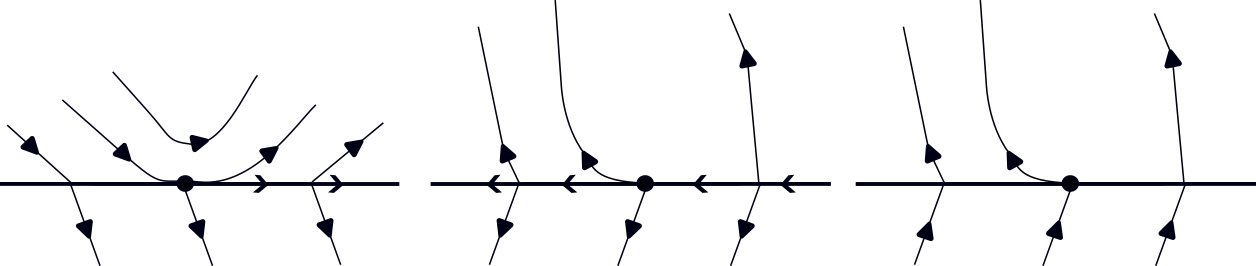}
		{\footnotesize
       \put(-60,-11){$(c)$}
       \put(-158,-11){$(b)$}
       \put(-258,-11){$(a)$}
}
         \end{center}
	\caption{(a) The sliding vector field is defined on one side of the tangency point. (b) The sliding vector field is defined on both sides of the tangency point. (c) A crossing region around the tangency point (in this case the slow divergence integral near the tangency point is not defined). }
	\label{fig-OneTangency}
\end{figure}

We can now define the slow divergence integral near the tangency point $p$.
\begin{definition}\label{def-SDIOneContact} Let $m_\lambda=[0,x_1]$. Then the slow divergence integral along $m_\lambda$ is defined as 
   \begin{equation}
       I(m_\lambda)=\lim_{x_0\to 0+}I([x_0,x_1]).\nonumber
       \end{equation}
\end{definition}
\begin{remark}
    The slow divergence integral $I(m_\lambda)$ from Definition \ref{def-SDIOneContact} is well-defined. Indeed, $\lim_{u\to\pm\infty} \phi'(u)=0$ (due to the smoothness of $\phi$ at $\pm \infty$ given after \eqref{phi-asim}). Moreover, we have (a) $\lim_{u\to 1-} \phi^{-1}(u)=+\infty$, (b) $\lim_{u\to 0+} \phi^{-1}(u)=-\infty$ (see \eqref{phi-asim}) and finally (c) $\frac{-Y_\lambda^-}{Y_\lambda^+-Y_\lambda^-}(x,0)$ tends to $1$ (resp. $0$) as $x\to 0+$ when $Z_\lambda^+(0)$ (resp. $Z_\lambda^-(0)$) is tangent to $\Sigma_\lambda$. It follows from (a), (b) and (c) that the integrand in \eqref{SDIPomoc} can be defined at $x=0$ ($0$ for $x=0$)  and that the integrand is continuous on the segment $m_\lambda$. This implies that $I(m_\lambda)$ is well-defined.
\end{remark}

\section{Limit cycles and fractal analysis through visible-invisible two-fold \textbf{$VI_3$}}\label{sectionVI3}
\subsection{Model and assumptions}\label{model+assumptions}
We consider a PWS system \eqref{zZpminvariance} where we assume that $\lambda\sim\lambda_0\in\mathbb R$, and $h_\lambda(x,y)=y$ (the switching boundary is the line $y=0$).\\
\\
\textbf{Assumption A.} Suppose that there are $\eta_-<0$ and $\eta_+>0$ such that the PWS system \eqref{zZpminvariance} for $\lambda=\lambda_0$ has stable sliding for all $x\in[\eta_-,0[$ (i.e., $Y_{\lambda_0}^+(x,0)<0$ and $Y_{\lambda_0}^-(x,0)>0$ for $x\in[\eta_-,0[$) and unstable sliding for all $x\in ]0,\eta_+]$ (i.e., $Y_{\lambda_0}^+(x,0)>0$ and $Y_{\lambda_0}^-(x,0)<0$ for $x\in]0,\eta_+]$). Moreover, we assume that the Filippov sliding vector field $X_\lambda^{sl}$ given by \eqref{FSVF}
is positive for $x\in [\eta_-,\eta_+]\setminus\{0\}$ and $\lambda=\lambda_0$.
\smallskip

Assumption A implies that $Y_{\lambda_0}^\pm(0)=0$ and the origin $z=0$ is therefore a tangency point (see Section \ref{section2general}). We assume that $z=0$ for $\lambda=\lambda_0$ is a two-fold singularity. Moreover, we suppose that the two-fold singularity is visible from ``above" and invisible from ``below", i.e., the orbit of $Z_{\lambda_0}^+$ through $z=0$ is contained within $y>0$ near $z=0$, and the orbit of $Z_{\lambda_0}^-$ through $z=0$ is not contained within $y<0$ (Section \ref{section2general}).
\\
\\
\textbf{Assumption B.} We assume that the origin $z=0$ in the PWS system \eqref{zZpminvariance} is a visible-invisible two-fold for $\lambda=\lambda_0$: $Y_{\lambda_0}^\pm(0)=0$ and 
 \begin{align}\label{twofoldcond}
  \begin{cases}
  X_{\lambda_0}^+ (0)&>0,\\ 
  \partial_xY_{\lambda_0}^+(0)&> 0,
  \end{cases}\quad  \begin{cases}
  X_{\lambda_0}^- (0)&< 0,\\ 
  \partial_x Y_{\lambda_0}^-(0)&<0.
  \end{cases}
  \end{align}
  Additionally, we assume that $(\det Z_{\lambda_0})'(0)<0$ where $\det Z_\lambda$ is defined in \eqref{FSVF}.

\begin{remark}
From \eqref{twofoldcond} it follows that $\partial_x (Y_{\lambda_0}^--Y_{\lambda_0}^+)(0)<0$. This, together with $(\det Z_{\lambda_0})'(0)<0$ and \eqref{removablesing}, implies that $X_{\lambda_0}^{sl}(0)>0$.  Thus, $X_{\lambda_0}^{sl}(x)>0$ for all $x\in [\eta_-,\eta_+]$ (see Assumption A).
\end{remark}

Assumption B and The Implicit Function Theorem imply the existence of smooth $\lambda$-families of fold singularities $z_+=(x_+(\lambda),0)$ from above and fold singularities $z_-=(x_-(\lambda),0)$ from below, for $\lambda\sim \lambda_0$, with $x_\pm(\lambda_0)=0$. The following assumption deals with non-zero velocity of the collision between $z_+$ and $z_-$ for $\lambda=\lambda_0$ at the origin $z=0$:
$$ x_+'(\lambda_0)-x_-'(\lambda_0)=\left(-\frac{\partial_{\lambda} Y_{\lambda_0}^+}{\partial_x Y_{\lambda_0}^+}+\frac{\partial_\lambda Y_{\lambda_0}^-}{\partial_x Y_{\lambda_0}^-}\right)(0)\ne 0$$
where $\partial_{\lambda} Y_{\lambda_0}^\pm$ means the partial derivative of $Y_{\lambda}^\pm$ w.r.t. $\lambda$, computed in $\lambda=\lambda_0$.
\\
\\
\textbf{Assumption C.} We assume that 
\begin{equation}
    \label{NZvelocity}
    \partial_\lambda Y_{\lambda}^- \partial_x Y_{\lambda}^+\ne \partial_\lambda Y_{\lambda}^+ \partial_x Y_{\lambda}^- 
\end{equation}
at $(z,\lambda)=(0,\lambda_0)$.
\smallskip

We consider a regularized PWS system \eqref{regul-gener} with $h_\lambda(x,y)=y$:
\begin{align}\label{regul-normal}
  \dot z &=\phi(y\epsilon^{-2}) Z_\lambda^+(z) + (1-\phi(y\epsilon^{-2}))Z_\lambda^-(z)
 \end{align}
where $0<\epsilon\ll 1$ and $\phi:\mathbb{R}\to \mathbb{R}$ is a smooth regularization function that satisfies the assumptions given after \eqref{regul-gener}. More precisely, we have  \\
 \\
 \textbf{Assumption D.} We suppose that $\phi$ satisfies \eqref{strictly-mono} and \eqref{phi-asim} and that $\phi$ is smooth at $\pm \infty$.

\smallskip

It is more convenient to write $\epsilon^{-2}$ in \eqref{regul-normal} instead of $\epsilon^{-1}$ so that we can directly use results from \cite{RHKK} (see Section \ref{section-proofs}).

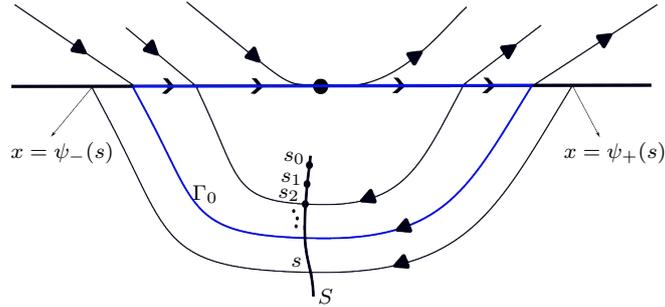
\begin{figure}[htb]
	\begin{center}
		\includegraphics[width=8.5cm,height=3.9cm]{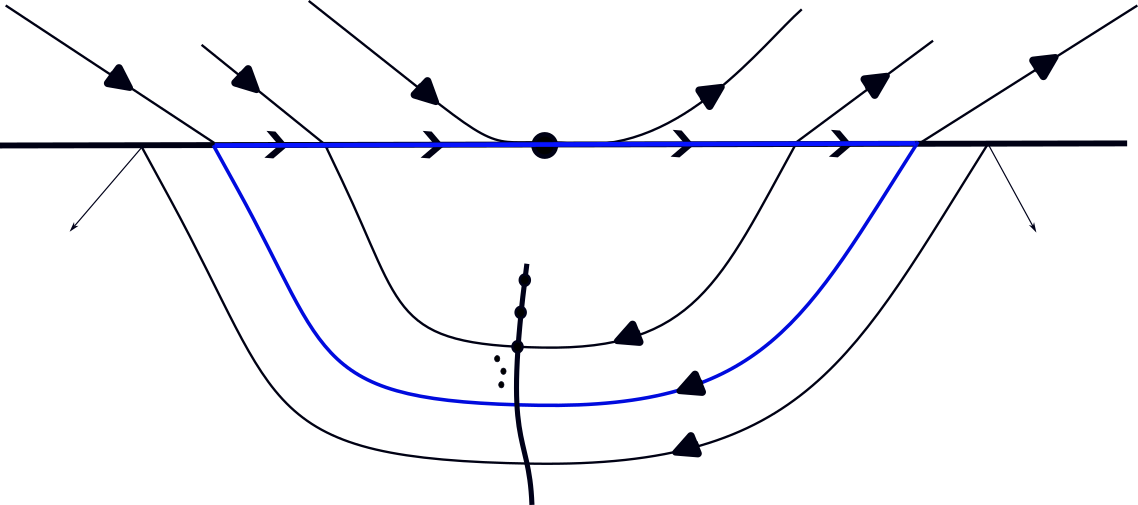}
		{\footnotesize
        \put(-140,51){$s_0$}
        \put(-141,43){$s_1$}
        \put(-142,37){$s_2$}
        \put(-174,36){$\Gamma_0$}
        \put(-137,11){$s$}
        \put(-127,-3){$S$}
        \put(-242,52){$x=\psi_-(s)$}
        \put(-35,52){$x=\psi_+(s)$}
}
         \end{center}
	\caption{A fractal sequence $(s_n)_{n\in\mathbb N}$ near the canard cycle $\Gamma_0$.}
	\label{fig-LPS}
\end{figure}
 \smallskip

Let $S$ be an open section transversally
cutting orbits of $Z_\lambda^-$, parametrized by a regular
parameter $s\sim 0$ (Fig. \ref{fig-LPS}). We assume that $s$ increases as we appraoch the origin $z=0$. For $\lambda=\lambda_0$, let $\Gamma_s$ be the limit periodic set consisting of the orbit of $Z_{\lambda_0}^-$ connecting $(\psi_+(s),0)$ and $(\psi_-(s),0)$, and the segment $[\psi_-(s),\psi_+(s)]\subset \{y=0\}$ (Fig. \ref{fig-LPS}). We suppose that $[\psi_-(s),\psi_+(s)]\subset [\eta_-,\eta_+]$ for all $s\sim 0$. In \cite{RHKK} $\Gamma_s$ is called a canard cycle. From the chosen parametrization of $S$ it follows that $\psi_-'(s)>0$ and $\psi_+'(s)<0$.
Following \cite[Section 3]{RHKK}, to study the number of limit cycles of \eqref{regul-normal} produced by $\Gamma_s$ for $(\epsilon,\lambda)\sim (0,\lambda_0)$ one can use the slow divergence integral associated to the segment $[\psi_-(s),\psi_+(s)]$:
\begin{equation}
    \label{SDI1}
    I(s)=\int_{\psi_-(s)}^{\psi_+(s)}\frac{(Y_{\lambda_0}^+-Y_{\lambda_0}^-)^2(x,0)}{-\det Z_{\lambda_0}(x)}\phi'\left(\phi^{-1}\left(\frac{-Y_{\lambda_0}^-}{Y_{\lambda_0}^+-Y_{\lambda_0}^-}(x,0)\right)\right)dx.
\end{equation}
\begin{remark}
In  \eqref{SDI1} we use Definition \ref{def-SDIContact} and Remark \ref{remark-SDI-negative}. Note that $$I(s)=I(m_{\lambda_0})+I(\tilde m_{\lambda_0})$$
where $m_{\lambda_0}=[0,\psi_+(s)]$ and $\tilde m_{\lambda_0}=[\psi_-(s),0]$.
\end{remark}
\begin{remark}\label{remark-results} We suppose that Assumptions A through D are satisfied and 
 write $$\lambda=\lambda_0+\epsilon\widetilde\lambda$$ with $\widetilde \lambda\sim 0$.
We say that the cyclicity of the canard cycle $\Gamma_{0}$ inside \eqref{regul-normal} is bounded by $N\in\mathbb N$ if there exist $\epsilon_0>0$, $\delta_0>0$ and a neighborhood $\mathcal U$ of $0$ in the $\widetilde{\lambda}$-space such that \eqref{regul-normal} has at most $N$
limit cycles, each lying within Hausdorff distance $\delta_0$
of $\Gamma_{0}$, for all $(\epsilon,\widetilde{\lambda})\in ]0,\epsilon_0]\times \mathcal U$. We call the
smallest $N$ with this property the cyclicity of $\Gamma_{0}$ and denote it by $\cycl(\Gamma_0)$. 

If the slow divergence integral $I$ in \eqref{SDI1} has a simple zero at $s=0$, then $\cycl(\Gamma_0)=2$ and, for each small $\epsilon>0$, the $\widetilde\lambda$-family in \eqref{regul-normal} undergoes a saddle-node bifurcation of
limit cycles near $\Gamma_0$ when we vary $\widetilde\lambda\sim 0$. Under the same assumption on $I$, there is a smooth function $\widetilde \lambda=\widetilde\lambda(\epsilon)$, $\widetilde\lambda(0)=0$, such that \eqref{regul-normal} with $\lambda=\lambda_0+\epsilon\widetilde\lambda(\epsilon)$ has a unique (hyperbolic) limit cycle Hausdorff close to $\Gamma_0$ for each small $\epsilon>0$.

If $I$ has a zero of multiplicity $m\ge 1$ at $s=0$, then $\cycl(\Gamma_0)\le m+1$. When $I(0)<0$ (resp. $I(0)>0$), then $\cycl(\Gamma_0)= 1$, and the limit cycle is hyperbolic and attracting (resp. repelling). 

We refer the reader to \cite[Theorem 3.1]{RHKK} and \cite[Remark 3.4]{RHKK}.
\end{remark}

We say that the canard cycle $\Gamma_0$ is balanced if $s=0$ is a zero of $I(s)$ defined in \eqref{SDI1} ($I(0)=0$). If $\Gamma_0$ is balanced, then there exists a unique function $G(s)$ satisfying $G(0)=0$, $G'(0)>0$ and 
\begin{equation}
    \label{SDISlowRelation}
    \int_{\psi_-(s)}^{\psi_+(G(s))}\frac{(Y_{\lambda_0}^+-Y_{\lambda_0}^-)^2(x,0)}{-\det Z_{\lambda_0}(x)}\phi'\left(\phi^{-1}\left(\frac{-Y_{\lambda_0}^-}{Y_{\lambda_0}^+-Y_{\lambda_0}^-}(x,0)\right)\right)dx=0
\end{equation}
for $s\sim 0$. This follows from The Implicit Function Theorem because $I(0)=0$, $\psi_-'(s)>0$, $\psi_+'(s)<0$ and the integrand in \eqref{SDI1} is negative for $x<0$ and positive for $x>0$ (see Assumptions A and D). We call $G$ defined by \eqref{SDISlowRelation} the slow relation function.
\\
\\
 \textbf{Assumption E.} We suppose that $\Gamma_0$ is balanced and that $s=0$ is an isolated zero of $I(s)$, meaning that $s=0$ has a small neighborhood $]-\tilde s,\tilde s[$ ($\tilde s>0$) that does not contain any other zero of $I(s)$.  

 Assumption E implies that $I$ is either negative or positive for $s>0$ ($I$ is continuous). Using the above mentioned property of the integrand in \eqref{SDI1} it can be easily seen that $0<G(s)<s$ for $s>0$ when $I$ is negative and $G(s)>s$ for $s>0$ when $I$ is positive. Let $s_0>0$ be small and fixed. Thus, if $I$ is negative (resp. positive), then the orbit of $s_0$
 \begin{equation}\label{orbit}
     U_0=\{s_0,s_1,s_2,\dots\}
 \end{equation}
defined by $s_{n+1}=G(s_n)$ (resp. $s_{n+1}=G^{-1}(s_n)$), $n\ge 0$, tends monotonically to the fixed point $s=0$ of $G$. We want to study the Minkowski dimension of $U_0$. 
\smallskip

Let us first define the notion of Minkowski (or box) dimension (see \cite{Falconer,tricot} and references therein).
Let $U\subset\mathbb R^N$ be a bounded set. We define the $\delta$-neighborhood of $U$:
$$
U_\delta=\{x\in\mathbb R^N \ | \ d(x,U)\le\delta\},$$ and denote by $|U_\delta|$ the Lebesgue measure of $U_\delta$.
The lower $u$-dimensional  Minkowski content of $U$, for $u\ge0$, is defined by
$$
\mathcal M_*^u(U)=\liminf_{\delta\to 0}\frac{|U_\delta|}{\delta^{N-u}},
$$
and analogously the upper $u$-dimensional Minkowski content $\mathcal M^{*u}(U)$ (we replace $\liminf_{\delta\to0}$ with $\limsup_{\delta\to0}$).
We define lower and upper Minkowski dimensions of $U$:
$$
\underline\dim_BU=\inf\{u\ge0 \ | \ \mathcal M_*^u(U)=0\}, \ \overline\dim_BU=\inf\{u\ge0 \ | \ \mathcal M^{*u}(U)=0\}.
$$
We have $\underline\dim_BU\le \overline\dim_BU$ and, if $\underline\dim_BU=\overline\dim_BU$, we call it the Minkowski dimension of $U$, and denote it by $\dim_BU$.
\smallskip

The upper Minkowski dimension is finitely stable. More precisely, $$\overline\dim_B(U_1\cup U_2)=\max\{\overline\dim_BU_1,\overline\dim_BU_2\}, \ U_1,U_2\subset \mathbb R^N.$$
If $U_1\subset U_2$, then $\underline\dim_BU_1\le \underline\dim_BU_2$ and $\overline\dim_BU_1\le \overline\dim_BU_2$ ($\underline\dim_B$ and $\overline\dim_B$ are monotonic).

Furthermore, if $0<\mathcal M_*^d(U)\leq\mathcal M^{*d}(U)<\infty$ for some $d$, then we say
 that $U$ is Minkowski nondegenerate. In this case we have necessarily that $d=\dim_B U$. 
Recall also that the notion of being Minkowski nondegenerate  is invariant under bi-Lipschitz maps.
Namely, if $\Phi$ is a bi-Lipschitz map and $U$ is Minkowski nondegenerate, then $\Phi(U)$ is also Minkowski nondegenerate (see \cite[Theorem 4.1]{ZuZuR^3}).
\smallskip

We use these properties in Section \ref{sec-proof3}.

Following \cite{EZZ}, $\dim_BU_0$ exists, it is independent of the choice of $s_0>0$ and  can take only the following discrete set of values: $0$, $\frac{1}{2}$, $\frac{2}{3}$, $\frac{3}{4}$,$\dots$, $1$ (see also Theorem \ref{th-EZZ}). The set $U_0$ is defined in \eqref{orbit}.

\subsection{Statement of results}\label{section-statement-fractal}
In this section we consider the family \eqref{regul-normal} that satisfies Assumptions A through E and assume that $\lambda=\lambda_0+\epsilon\widetilde\lambda$ with $\widetilde \lambda\sim 0$. 

\begin{theorem}\label{theorem-1}
Let $s_0>0$ be small and fixed and let $U_0$ be the orbit of $s_0$ defined in \eqref{orbit}. If $\dim_BU_0=0$, then the following statements hold.
\begin{enumerate}
\item ($\lambda$ unbroken) There exists a smooth function $\widetilde\lambda=\widetilde\lambda(\epsilon)$, $\widetilde\lambda(0)=0$, such that \eqref{regul-normal} with $\lambda=\lambda_0+\epsilon\widetilde\lambda(\epsilon)$ has a unique (hyperbolic) limit cycle Hausdorff close to $\Gamma_0$ for each small $\epsilon>0$.
    \item ($\lambda$ broken) We have that  $\cycl(\Gamma_0)=2$ and, for every small $\epsilon>0$, the $\widetilde\lambda$-family \eqref{regul-normal} undergoes a saddle-node bifurcation of limit cycles Hausdorff close to $\Gamma_0$. 
\end{enumerate}
\end{theorem}
Theorem \ref{theorem-1} will be proved in Section \ref{sec-proofs12}.

\begin{theorem}\label{theorem-2}
 Let $U_0$ be the orbit of $s_0$ defined in \eqref{orbit}, for a small $s_0>0$. If $\dim_BU_0<1$, then $\cycl(\Gamma_0)\le \frac{2-\dim_BU_0}{1-\dim_BU_0}$.  
\end{theorem}
Theorem \ref{theorem-2} will be proved in Section \ref{sec-proofs12}.

\begin{theorem}\label{theorem-3}
Let $U_0$ be the orbit of $s_0$ defined in \eqref{orbit}, for a small $s_0>0$, and $\dim_BU_0=0$. The following statements are true.
\begin{enumerate}
\item For $\widetilde\lambda=\widetilde\lambda(\epsilon)$ given in Theorem \ref{theorem-1}.1 and for each small $\epsilon>0$, the Minkowski dimension of any spiral trajectory accumulating (in forward or backward time) on the unique limit cycle of \eqref{regul-normal} near $\Gamma_0$ is equal to $1$.
    \item For each small $\epsilon>0$, the Minkowski dimension of any spiral trajectory accumulating (in forward or backward time) on the limit cycle of multiplicity $2$ of \eqref{regul-normal}, born in a saddle-node bifurcation of limit cycles Hausdorff close to $\Gamma_0$, is equal to $\frac{3}{2}$ and moreover,  the spiral is Minkowski nondegenerate. 
\end{enumerate}
\end{theorem}
Theorem \ref{theorem-3} will be proved in Section \ref{sec-proof3}. A small (Hausdorff) neighborhood of $\Gamma_0$ in which we consider spiral trajectories in Theorem \ref{theorem-3}.1 or Theorem \ref{theorem-3}.2 does not shrink to $\Gamma_0$ as $\epsilon\to 0$ (see Section \ref{sec-proof3}).

\section{Proof of Theorems \ref{theorem-1}--\ref{theorem-3}}\label{section-proofs}

\subsection{Proof of Theorems \ref{theorem-1}--\ref{theorem-2}} \label{sec-proofs12}
Let $\tilde s>0$ be small and fixed. Suppose that $F$ is a smooth function on $[0,\tilde s[$, $F(0)=0$ and $0<F(s)<s$ for all $s\in ]0,\tilde s[$. We define $H(s):=s-F(s)$ and the orbit of $s_0\in ]0,\tilde s[$ by $H$: $$U:=\{s_n=H^n(s_0) \ | \ n=0,1,\dots\}$$
where $H^n$ denotes $n$-fold composition of $H$. It is clear that $s_n$ tends monotonically to zero. We say that the multiplicity of the fixed point $s=0$ of $H$ is equal to $m$ if $s=0$ is a zero of multiplicity $m$ of $F$ ($F(0)=\cdots=F^{(m-1)}(0)=0$ and $F^{(m)}(0)\ne 0$). If $F^{(n)}(0)=0$ for each $n=0,1,\dots$, then we say that the multiplicity of $s=0$ of $H$ is $\infty$.
\begin{theorem}[\cite{EZZ}]\label{th-EZZ}
 Let $F$ be a smooth function on $[0,\tilde s[$, $F(0)=0$ and $0<F(s)<s$ for each $s\in ]0,\tilde s[$. Let $H=id-F$ and let $U$ be the orbit of $s_0\in ]0,\tilde s[$ by $H$. Then $\dim_BU$ is independent of the initial point $s_0$ and, for $1\le m\le \infty$, the following bijective correspondence holds:
 \begin{equation}\label{bijective}
     m=\frac{1}{1-\dim_BU}
 \end{equation}
 where $m$ is the multiplicity of $s=0$ of $H$ (if $m=\infty$, then $\dim_BU=1$).
\end{theorem}

If we denote by $\Phi$ the integrand in \eqref{SDI1} and \eqref{SDISlowRelation}, then we have
$$ I(s)=\int_{\psi_-(s)}^{\psi_+(s)}\Phi(x)dx=\int_{\psi_-(s)}^{\psi_+(G(s))}\Phi(x)dx+\int_{\psi_+(G(s))}^{\psi_+(s)}\Phi(x)dx=\int_{\psi_+(G(s))}^{\psi_+(s)}\Phi(x)dx$$
where in the last step we use \eqref{SDISlowRelation}. From $\psi_+'(s)<0$, $\Phi(x)>0$ for $x>0$ and The Fundamental Theorem of Calculus it follows that there exists a negative smooth function $\Psi(s)$ such that 
$$I(s)=\Psi(s)(s-G(s)).$$
This implies that $s=0$ is a zero of multiplicity $m$ of $I(s)$ if and only if $s=0$ is a zero of multiplicity $m$ of $s-G(s)$.

We will first suppose that the orbit $U_0$ in \eqref{orbit} is generated by the slow relation function $G$. If $\dim_BU_0=0$, then Theorem \ref{th-EZZ}, with $H=G$, implies that  the multiplicity of the fixed point $s=0$ of $G$ is $1$. Thus, we have that $s=0$ is a simple zero of $I$ and Theorem \ref{theorem-1}.1 (resp. Theorem \ref{theorem-1}.2) follows directly from \cite[Theorem 3.1]{RHKK} (resp. \cite[Remark 3.4]{RHKK}). See also Remark \ref{remark-results}. If $\dim_BU_0<1$, then the multiplicity of $s=0$ of $G$ is equal to $\frac{1}{1-\dim_BU_0}$ (see \eqref{bijective}). Thus, $s=0$ is a zero of multiplicity $\frac{1}{1-\dim_BU_0}$ of $I$ and \cite[Remark 3.4]{RHKK}) implies that $$\cycl(\Gamma_0)\le 1+\frac{1}{1-\dim_BU_0}=\frac{2-\dim_BU_0}{1-\dim_BU_0}.$$ This completes the proof of Theorem \ref{theorem-2}.

If $U_0$ is generated by $G^{-1}$, Theorem \ref{theorem-1} and Theorem \ref{theorem-2} can be proved in the same way as above (we use Theorem \ref{th-EZZ} with $H=G^{-1}$ and the fact that $G$ and $G^{-1}$ have the same multiplicity of the fixed point $s=0$).

\subsection{Proof of Theorem \ref{theorem-3}}\label{sec-proof3}
The Minkowski dimension of spiral trajectories accumulating on a hyperbolic or non-hyperbolic limit cycle of planar vector fields (without parameters) has been studied in \cite{zubzup05,ZupZub}. We prove Theorem \ref{theorem-3} for spiral trajectories in a Hausdorff neighborhood of the canard cycle $\Gamma_0$ that does not shrink to $\Gamma_0$ as $\epsilon\to 0$.

We will first prove Theorem \ref{theorem-3}.1. We assume that $\dim_BU_0=0$ and $\widetilde\lambda(\epsilon)$ is given in Theorem \ref{theorem-1}.1. Let $\bar V$ be a fixed neighborhood of $\Gamma_0$. Then the unique limit cycle of \eqref{regul-normal} with $\widetilde\lambda=\widetilde\lambda(\epsilon)$ is located in $\bar V$ for each $\epsilon>0$ small enough (see \cite{RHKK}). For such fixed $\epsilon>0$, let $\Gamma$ be any spiral trajectory in $\bar V$ accumulating on the limit cycle (in the forward time if the limit cycle is attracting or in the backward time if the limit cycle is repelling). We write $\Gamma=\tilde \Gamma\cup \bar\Gamma$ where $\bar\Gamma$ is the part of $\Gamma$ sufficiently close to the limit cycle (we can apply the results of \cite{zubzup05,ZupZub}) and $\tilde \Gamma$ is the rest of $\Gamma$ (of finite length). It is clear that $\dim_B\tilde \Gamma=1$ and $\overline\dim_B \bar\Gamma\ge 1$. 
Since $\underline\dim_B\le \overline\dim_B$, $\underline\dim_B$ is monotonic and $\overline\dim_B$ is finitely stable (see Section \ref{model+assumptions}), we have 
\begin{equation}\label{inequ}
\underline\dim_B \bar\Gamma\le \underline\dim_B (\tilde \Gamma\cup \bar\Gamma) \le \overline\dim_B (\tilde \Gamma\cup \bar\Gamma)=\max\{\overline\dim_B \tilde \Gamma,\overline\dim_B \bar\Gamma\}=\overline\dim_B \bar\Gamma.
\end{equation}
Since the limit cycle is hyperbolic (see Theorem \ref{theorem-1}.1), \cite[Theorem 10]{zubzup05} implies that $\dim_B\bar\Gamma=\underline\dim_B \bar\Gamma=\overline\dim_B \bar\Gamma=1$. Using \eqref{inequ} we obtain $\dim_B\Gamma=1$. This completes the proof of Theorem \ref{theorem-3}.1.

The first part of Theorem \ref{theorem-3}.2 can be proved in
the same way as Theorem \ref{theorem-3}.1. Since the non-hyperbolic limit cycle is generated by a saddle-node bifurcation of limit cycles we have $\dim_B\bar\Gamma=\underline\dim_B \bar\Gamma=\overline\dim_B \bar\Gamma=\frac{3}{2}$ (see \cite[Theorem 10]{zubzup05} and \cite[Theorem 1]{ZupZub}).
To prove the claim about Minkowski nondegeneracy; first observe that $\mathcal{M}^{3/2}({\tilde\Gamma})=0$ since $\dim_B(\tilde\Gamma)=1<3/2$ so that this part does not affect the upper and lower Minkowski content of $\Gamma=\tilde \Gamma\cup \bar\Gamma$; hence, it is enough to show that $\bar\Gamma$ is Minkowski nondegenerate. 
To see this, we observe that $\bar\Gamma$ can be partitioned into finitely many pieces $\bar\Gamma_i$; $i=1,\ldots,k$ such that each $\bar\Gamma_i$ is bi-Lipschitz equivalent to $[0,1[\times U$ by the Lipschitz flow-box Theorem \cite{CALCATERRA20081108}.
Note also that $\dim_BU=1/2$ and it is Minkowski nondegenerate which implies that $[0,1[\times U$ is also Minkowski nondegenerate; see the proof of \cite[Theorem 4(b)]{ZupZub}.
Finally, the finite stability of Minkowski dimension and of Minkowski nondegeneracy now complete the proof exactly as in the proof of \cite[Theorem 4(b)]{ZupZub}.

\section*{Declarations}
 
\textbf{Ethical Approval} \ 
Not applicable.
 \\
\\
 \textbf{Competing interests} \  
The authors declare that they have no conflict of interest.\\
 \\
\textbf{Authors' contributions} \  All authors conceived of the presented idea, developed the theory, performed the computations and
contributed to the final manuscript.  \\ 
\\ 
\textbf{Funding} \
The research of R. Huzak and G. Radunovi\'{c}  was supported by: Croatian Science Foundation (HRZZ) grant
PZS-2019-02-3055 from “Research Cooperability” program funded by the European Social Fund. Additionally, the research of G. Radunovi\'{c} was partially
supported by the HRZZ grant UIP-2017-05-1020.\\
 \\
\textbf{Availability of data and materials}  \
Not applicable.

\bibliographystyle{plain}
\bibliography{bibtex}
\end{document}